\documentclass[12pt]{amsart}
\usepackage{amssymb,amsmath,amsthm}
\usepackage{epsfig}
\usepackage{graphicx}

\newtheorem{theorem}{Theorem}[section]
\newtheorem{lemma}[theorem]{Lemma}

\newtheorem{proposition}[theorem]{Proposition}

\long\def\symbolfootnote[#1]#2{\begingroup%
\def\thefootnote{\fnsymbol{footnote}}\footnote[#1]{#2}\endgroup}

\begin{document}

\title{Embedding Groups into Distributive Subsets of the Monoid of Binary Operations}
\author{Gregory T.~Mezera}


\centerline{February 2012- October 2012}
\subjclass{Primary 55N35; Secondary 18G60, 57M25}
\keywords{monoid of binary operations, distributive set, shelf, multi-shelf, distributive homology}

\thispagestyle{empty}

\begin{abstract}
Let $X$ be a set and $Bin(X)$ the set of all binary operations on $X$.
We say that $S\subset Bin(X)$ is a distributive set of operations if all
pairs of elements $*_{\alpha},*_{\beta} \in S$ are right distributive,
 that is, $ (a*_{\alpha}b)*_{\beta}c= (a*_{\beta}c)*_{\alpha}(b*_{\beta}c)$ (we allow  $*_{\alpha}=*_{\beta}$).

J.Przytycki posed the question of which groups can be realized as distributive sets. The initial guess
that any group may be  embedded into $Bin(X)$ for some $X$ was complicated by an observation that
if $*\in S$ is idempotent ($a*a=a$), then $*$ commutes with every element of $S$. The first
noncommutative subgroup of $Bin(X)$ (the group $S_3$) was found in October of 2011 by Y.Berman.

We show that any group can be embedded in $Bin(X)$ for $X=G$ (as a set).
We also discuss minimality of embeddings observing that, in particular, $X$ with six elements 
is the smallest set such that $Bin(X)$ contains a non-abelian subgroup.

\end{abstract}
\maketitle

\tableofcontents

\section{Introduction}
Let $X$ be a set and $Bin(X)$ the set of all binary operations on $X$. 
We say that $S\subset Bin(X)$ is a distributive set of operations if  all
pairs of elements $*_{\alpha},*_{\beta} \in S$ are right distributive,
 that is, $ (a*_{\alpha}b)*_{\beta}c= (a*_{\beta}c)*_{\alpha}(b*_{\beta}c)$ (we allow  $*_{\alpha}=*_{\beta}$).
$(X;S)$ is called a multi-shelf\footnote{If $(X;*)$ is a magma and $*$ is a right self-distributive operation, 
then $(X;*)$ is called a shelf - the term coined by Alissa Crans in her PhD thesis \cite{Cr}.} in this case.
It was observed in \cite{Prz-1} (compare also \cite{Ro-S}) that $Bin(X)$ is a monoid with  composition
$*_1*_2$ given by $a*_1*_2b= (a*_1b)*_2b$, with the identity $*_0$ being the right trivial operation, that is,
$a*_0b=a$ for any $a,b\in X$.

The submonoid of $Bin(X)$ of all invertible elements in $Bin(X)$ is a group denoted by $Bin_{inv}(X)$.
If $* \in Bin_{inv}(X)$ then $*^{-1}$ is usually denoted by $\bar *$.

We say that a subset $S \subset Bin(X)$ is a distributive set if all
pairs of elements $*_{\alpha},*_{\beta} \in S$ are right distributive,
 that is, $ (a*_{\alpha}b)*_{\beta}c= (a*_{\beta}c)*_{\alpha}(b*_{\beta}c)$ (we allow  $*_{\alpha}=*_{\beta}$).

The following important basic lemma was proven in \cite{Prz-1}:
\begin{lemma}\label{Lemma 1}
\begin{enumerate}
\item[(i)] If $S$ is a distributive set and $*\in S$ is invertible, then $S\cup \{\bar *\}$ is also a
distributive set.
\item[(ii)] If $S$  is a distributive set and $M(S)$ is the monoid generated by $S$, then $M(S)$ is a
distributive monoid.
\item[(iii)] If $S$  is a distributive set of invertible operations and $G(S)$ is the group generated by $S$, then
$G(S)$ is a distributive group.
\end{enumerate}
\end{lemma}

The question was asked by J.Przytycki which groups can be realized as distributive sets. 
Soon after the definition of a distributive submonoid of $Bin(X)$ was given in \cite{Prz-1},  
Michal Jablonowski, a graduate student
at Gda\'nsk University, noticed that any distributive monoid whose elements are idempotent operations
is commutative. We have:
\begin{proposition} (\cite{Prz-1})\label{Proposition ??}
\begin{enumerate}
\item[(i)]
 Consider $*_{\alpha},*_{\beta}\in Bin(X)$ such that $*_{\beta}$ is idempotent ($a*_{\beta}a=a$) and
distributive with respect to $*_{\alpha}$. Then $*_{\alpha}$ and $*_{\beta}$ commute. In particular:
\item[(ii)] If $M$ is a distributive monoid and $*_{\beta}\in M$ is an idempotent operation, then $*_{\beta}$
is in the center of $M$.
\item[(iii)] A distributive monoid whose elements are idempotent operations is commutative.
\end{enumerate}
\end{proposition}
\begin{proof} We have: $(a*_{\alpha}b)*_{\beta}b \stackrel{distrib}{=} (a*_{\beta}b)*_{\alpha}(b*_{\beta}b)
\stackrel{idemp}{=}
(a*_{\beta}b)*_{\alpha}b$.
\end{proof}
A few months later Agata Jastrz{\c e}bska (also  a graduate student at Gda\'nsk University) 
checked that any distributive group in $Bin_{inv}(X)$ for  $|X|\leq 5$  is
commutative. 

The first 
noncommutative subgroup of $Bin(X)$ (the group $S_3$) was found in October of 2011 by Yosef Berman.
Soon after Berman (with the help of Carl Hammarsten) constructed an embedding of a general dihedral group $D_{2\cdot n}$
in $Bin(X)$ where $X$ has $2n$ elements. The embedding of Berman $\phi: D_{2\cdot 3}\to Bin(X)$ is 
given as follows:
if $X=\{0,1,2,3,4,5\}$ then the subgroup $D_{2\cdot 3} \subset Bin(X)$ is generated by the binary operations
$*_{\tau}$ (reflection) and $*_{\sigma}$ (a $3$-cycle):
$$*_{\tau}=  \left( \begin{array}{cccccc}
1 & 1 & 3 & 5 & 5 & 3 \\
0 & 0 & 4 & 2 & 2 & 4 \\
3 & 3 & 5 & 1 & 1 & 5 \\
2 & 2 & 0 & 4 & 4 & 0  \\
5 & 5 & 1 & 3 & 3 & 1 \\
4 & 4 & 2 & 0 & 0 & 2
\end{array} \right)
\text{ and }
*_{\sigma} =  \left( \begin{array}{cccccc}
2 & 4 & 2 & 4 & 2 & 4 \\
5 & 3 & 5 & 3 & 5 & 3 \\
4 & 0 & 4 & 0 & 4 & 0 \\
1 & 5 & 1 & 5 & 1 & 5 \\
0 & 2 & 0 & 2 & 0 & 2 \\
3 & 1 & 3 & 1 & 3 & 1
\end{array} \right). $$
where $i*j$ is placed in the $i$th row and $j$th column, 
and $D_{2\cdot 3}=\{\tau, \sigma\ | \ \tau \sigma \tau = \sigma^{-1} \}$.


\section{Regular distributive embedding}

We now show that any group $G$ can be embedded in $Bin(X)$ for some $X$.
\begin{theorem}(Regular embedding)\label{Theorem 1}\\
Every group $G$ embeds in $Bin(G)$.
This embedding (monomorphism), $\phi^{reg}: G \rightarrow Bin(G)$ sends $g$ to $*_g$ where $a*_gb=ab^{-1}gb$.
\end{theorem}
\begin{proof}
(i)
We check that the set $\{*_g\}_{g\in G}$ is a distributive set.
We have: \\
$(a*_{g_1}b)*_{g_2}c= (ab^{-1}g_1b)*_{g_2}c= ab^{-1}g_1bc^{-1}g_2c$,
and 
$$(a*_{g_2}c)*_{g_1}(b*_{g_2}c)=(ac^{-1}g_2c)*_{g_1}(bc^{-1}g_2c)= 
ab^{-1}g_1bc^{-1}g_2c, \mbox{ as needed.}$$
(ii) Now we check that the map $\phi^{reg}$ is a monomorphism. 
Of course the image of the identity $*_0$ is the identity in $Bin(G)$. Furthermore:
$a*_{g_1g_2}b=ab^{-1}g_1g_2b$, and \\
$a*_{g_1}*_{g_2}b= (a*_{g_1}b)*_{g_2}b=ab^{-1}g_1bb^{-1}g_2b=ab^{-1}g_1g_2b$, as needed.
We have proven that $\phi^{reg}$ is a homomorphism. To show that $\phi^{reg}$ is a monomorphism 
we substitute  $b=1$ in the formula for $a*_gb$, to get
$a*_{g} 1=ag$, so different $g$'s give different binary operations in $Bin(G)$.
Notice that $\phi^{reg}(g^{-1}) = \bar *_g$.
\end{proof}

We call our embedding {\it regular} by analogy to the regular representation of a group.
We do not claim that the regular embedding is minimal. In fact, finding minimal distributive embeddings
is a very interesting problem in itself. 

\section{General conditions for a distributive embedding}

We now discuss a method that can be used to embed groups into subsets of $Bin_{inv}(X)$ 
satisfying a given condition. We then use this method when the condition is right distributivity,
which led us to discover the regular distributive embedding of $G$ in $Bin(G)$, 
and also should be a natural tool to look for minimal embeddings. For the group $S_3$ 
we know, by Jastrzebska's calculations, that  $X$ with six elements is the minimal set such that 
$S_3$ embeds in $Bin(X)$.

We start from the following basic observation:
\begin{lemma}
There is an isomorphism between $Bin_{inv}(X)$ and $S_{X}^{|X|}$, where $|X|$ is the cardinality of $|X|$
and $S_{X}$ is the group of permutation on set $X$  ( i.e. bijections of the set $X$). The isomorphism 
$\alpha: Bin_{inv}(X) \to S_{X}^{|X|}=\Pi_{y\in X}{S}_X^y$ is described as follows:
 $\alpha (*)(y): X\to X$ is the bijection  where 
$(\alpha (*)(y))(x)=x*y$. In other words $\alpha (*)(y)$ is the bijection corresponding to the $y$ coordinate of 
$S_{X}^{|X|}$. 
\end{lemma}

Using the map $\alpha $, we can translate conditions on a set of binary
operations in $Bin(X)$ into a group-theoretic condition on (coordinates of)
elements of $S_{X}^{|X|}.$ With some work, we can use this to find an embedding
of a group into $Bin(X).$ This is possible since the group axioms
require that such an embedding sits inside $Bin_{inv}(X).$ 
Let us consider distributive, invertible sets $\mathcal{S}$ of
binary operations in $Bin_{inv}(X)$. These are subsets $\mathcal{S}\subseteq
Bin_{inv}(X)$ that satisfy:
\[
(x\ast _{i}y)\ast _{j}z=(x\ast _{j}z) \ast _{i}(y\ast _{j}z),\text{ }for\text{ }all\text{
}\ast _{i},\ast _{j}\in S\text{ }and\text{ }x,y,z\in X.
\]

Let $\sigma _{i,y}=p_{y}\alpha (*_i),$ where $p_{y}:S_{X}^{|X|}\rightarrow
S_{X}$ is projection onto the $y^{th}$ coordinate. Then translating the
distributivity condition via $\alpha $:
\[
\sigma _{j,z}(x\ast _{i}y)=\sigma _{i,(y\ast _{j}z)}(x*_jz),
\]%
or%
\[
\sigma _{j,z}(\sigma _{i,y}(x))=\sigma _{i,\sigma _{j,z}(y)}(\sigma_{j,z}(x)),
\]%
which leads to%
\[
\sigma _{i,\sigma _{j,z}(y)}=\sigma _{j,z}\sigma _{i,y}\sigma _{j,z}^{-1}.
\]

Now the problem of embedding a group into $Bin_{inv}(X)$ is reduced to finding subsets of $S_{X}^{|X|}$
 isomorphic to the group that satisfy the condition above.
 We can then use tools of group theory
(e.g., representation theory) to solve the problem.
This process can by attempted for subsets of $Bin_{inv}(X)$ satisfying any condition, and led 
to the embedding defined in the previous section for distributive subsets.

\section{Future directions; multi-term homology}
Przytycki defined multi-term homology for any distributive 
set  in \cite{Prz-1}. This provided motivation to have examples of 
distributive sets. The regular embedding of a group (Theorem \ref{Theorem 1}) provides an 
interesting family of distributive sets ripe for studying this homology (compare \cite{CPP,Prz-1,Prz-2,Pr-Pu,P-S}).
As a nontrivial example we propose computing $n$-term distributive homology related with 
the regular embedding of the cyclic group $Z_n$.
Another problem related to Theorem \ref{Theorem 1} is which
 monoids are distributive submonoids of $Bin(X)$.

A key motivation is to use multi-term distributive homology in knot theory. This possibility arises from 
 the relation of the third Reidemeister move  with right distributivity (and eventually the Yang-Baxter 
operator), and the important 
work of Carter, Kamada, Saito, and other researchers on applications of quandle homology 
to knot theory (see \cite{CKS}).

\section{Acknowledgements}
I was partially supported by the GWU  Presidential Merit Fellowship.\\
I would like to thank Prof. Jozef Przytycki,  Carl Hammarsten, and Krzysztof Putyra for helpful discussions.

\ \\
\newpage
\ \\
Gregory T.~Mezera\\
Address:\\
Department of Mathematics,\\
George Washington University,\\
e-mail:   gtm@gwmail.gwu.edu
\end{document}